\documentclass[preprint, 10pt, english]{elsarticle}
\usepackage{amsthm}
\usepackage{amsmath}
\usepackage{latexsym, amssymb}
\usepackage{txfonts}
\usepackage{mathtools}
\usepackage{color}
\usepackage{babel}
\usepackage[all]{xy}
\usepackage[capitalise]{cleveref}

\newtheorem{thm}{Theorem}[section] 

\newtheorem{cor}[thm]{Corollary}

\newtheorem{lem}[thm]{Lemma}
\newtheorem{prop}[thm]{Proposition}

\theoremstyle{definition}
\newtheorem{rem}[thm]{Remark}
\newtheorem{exmpl}[thm]{Example}

\newcommand\operA[2]{{\if!#2!\operatorname{#1}\else{\operatorname{#1}_{#2}^{\phantom{I}}}\fi}} 

\newcommand{\Trace}[1][]{\if!#1!\operatorname{Tr}\else{\operatorname{Tr}_{#1}^{\phantom{I}}}\fi} 

\long\def\forget#1\forgotten{{}} %

\def\({\left(}
\def\){\right)}

\newcommand\LAY[3][]{{\begin{array}{c}\mbox{#2} \if#1!{}\else{+}\fi \\ \mbox{#3}\end{array}}}

\makeatletter

\def\ps@pprintTitle{%
 \let\@oddhead\@empty
 \let\@evenhead\@empty
 \def\@oddfoot{}%
 \let\@evenfoot\@oddfoot}

\newcommand{\bigperp}{%
  \mathop{\mathpalette\bigp@rp\relax}%
  \displaylimits
}

\newcommand{\bigp@rp}[2]{%
  \vcenter{
    \m@th\hbox{\scalebox{\ifx#1\displaystyle2.1\else1.5\fi}{$#1\perp$}}
  }%
}
\makeatother

\newif\iffurther
\furtherfalse

\journal{??}

\begin{document}
\begin{frontmatter}

\title{The Cyclicity of Tensor Products of Cyclic $p$-Algebras}

\author{Adam Chapman}
\address{School of Computer Science, Academic College of Tel-Aviv-Yaffo, Rabenu Yeruham St., P.O.B 8401 Yaffo, 6818211, Israel}
\ead{adam1chapman@yahoo.com}

\begin{abstract}
We revisit the famous theorem of Albert's on the cyclicity of tensor products of cyclic $p$-algebras. In the case of tensor products of cyclic $p$-algebras of prime degree, we provide an explicit computation of the resulting cyclic algebra in symbol algebra terms.
\end{abstract}

\end{frontmatter}

\section{Introduction}

It is known (\cite[Page 109, Lemma 13]{Albert:1968}) that any tensor product of cyclic $p$-algebras is cyclic. 
The original proof focuses on the existence of a purely inseparable simple extension of the center of appropriate degree, but does not provide an explicit formula for the resulting cyclic algebra.

In this note, we provide a computational method for describing the resulting cyclic algebra in terms of symbol presentation. The trick is an extension of an observation from \cite{ChapmanQueguiner:2024} on symbol presentation with fewest possible independent entries.

\section{Central Simple Algebras}

Recall that a central simple algebra over $F$ is a simple algebra $A$ whose center is $F$ and its dimension over $F$ is finite. Such algebras are always of dimension $d^2$ for some natural number $d$ called the degree of $A$.
Every such algebra decomposes as $D \otimes_F M_k(F)$ where $D$ is a division algebra with $Z(D)=F$. Since $D$ is also a central simple $F$-algebra, this gives rise to the notion of Brauer equivalence: the class of $A$ is represented by $D$.
The split case is the equivalence class of $F$ itself (which includes all matrix algebras over $F$).
The index of $A$ is the degree of $D$.
The set of Brauer classes forms a group, called the Brauer group, with $\otimes_F$ as the group binary operation.
This group is a torsion group. The order of each $A$ is called its exponent (or order or period).
The order always divides the index which in term divides the degree. See \cite{GilleSzamuely:2006}.

\section{Cyclic $p$-Algebras}

A $p$-algebra is a central simple algebra of degree $p^m$ over a field $F$ of $\operatorname{char}(F)=p>0$ where $m$ is some natural number. Such an algebra is cyclic if it contains a cyclic field (or \'{e}tale) extension of the center of degree $p^m$. These algebras are known to generate the $p^m$ torsion of the Brauer group of $F$ (see \cite[Section 9]{GilleSzamuely:2006} or \cite[Chapter 7]{Albert:1968}).

Every such cyclic algebra admits a symbol presentation $[\omega,\beta)_{p^m,F}$, which means the algebra is generated over $F$ by $u_1,\dots,u_m$ and $y$ satisfying $(u_1^p,\dots,u_m^p)-(u_1,\dots,u_m)=\omega$, $y^{p^m}=\beta$ and $(y u_1 y^{-1},\dots,y u_m y^{-1})=(u_1,\dots,u_m)+(1,0,\dots,0)$, where $\omega \in W_m F$ ($W_m F$ is the ring of truncated Witt vectors of length $m$ over $F$), $\beta \in F^\times$ and addition and subtraction follow the formulas of Witt vectors. See for instance \cite{MammoneMerkurjev:1991}.

These symbol presentations obey certain identities.
In particular, 
\begin{itemize}
\item $[(\beta,0,\dots,0),\beta)_{p^m,F}=0$, 
\item $[\omega,\beta)_{p^m,F}=[\omega+(\tau_1^p,\dots,\tau_m^p)-(\tau_1,\dots,\tau_m),\beta)_{p^m,F}$ for any $(\tau_1,\dots,\tau_m) \in W_m F$,
\item $[\omega,\beta)_{p^m,F}=[\omega,\pi \beta)_{p^m,F}$ for any norm $\pi$ in the field extension $F[u_1,\dots,u_m]/F$,
\item $[\omega,\beta)_{p^m,F}=[\omega,\beta^p)_{p^{m+1},F}=[(0,\omega),\beta)_{p^{m+1},F}$,
\item $[\omega,\beta)_{p^m,F} \otimes [\tau,\beta)_{p^m,F}=[\omega+\tau,\beta)_{p^m,F}$,
\item  and $[\omega,\beta)_{p^m,F} \otimes [\omega,\gamma)_{p^m,F}=[\omega,\beta \gamma)_{p^m,F}$. 
\end{itemize}
Note that the equality sign here stands for Brauer equivalence and $\otimes$ is always over $F$.

\begin{prop}
For any $\omega \in W_m F$, $\beta \in F^\times$ and $x\in F$, unless $\beta+x^{p^m}=0$, we have $[\omega,\beta)_{p^m,F}=[\omega \cdot (1+x^{p^m}\beta^{-1},0,\dots,0)\cdot \pi,\beta+x^{p^m})_{p^m,F}$ where $\pi$ is the unique element in $W_m F$ satisfying $\pi \cdot d(\beta+x^{p^m},0,\dots)=d(\beta,0,\dots,0)$ in $W_m \Omega_F^1$.
\end{prop}

\begin{proof}
From \cite[Theorem 2.27]{AravireJacobORyan:2018} we obtain an epimorphism from $W_m \Omega^1_F$ to the $p^m$ torsion of $Br(F)$ given by mapping $\omega \cdot (\beta^{-1},0,\dots,0) d(\beta,0,\dots,0)$ to $[\omega,\beta)_{p^m,F}$.
The statement then follows from the equality 
$\omega \cdot (\beta^{-1},0,\dots,0) d(\beta,0,\dots,0)=\omega \cdot (\beta^{-1},0,\dots,0) \cdot (\beta+x^{p^m},0,\dots,0) \cdot ((\beta+x^{p^m})^{-1},0,\dots,0) \cdot \pi d(\beta+x^{p^m},0,\dots,0).$
\end{proof}

\begin{rem}
This $\pi$ is a Witt vector that starts with 1, and therefore, when $m=1$, it can be erased.
However, when $m>1$, it is not necessarily the vector $(1,0,\dots,0)$. For example, when $p=m=2$, $d((\beta,0)+(x^4,0))=d(\beta,0)+d(x^4,0)=d(\beta,0)$ on the one hand, and on the other hand we have $d((\beta,0)+(x^4,0))=d(\beta+x^4,\beta x^4)=d(\beta+x^4,0)+d(0,\beta,x^4)$.
Now, $d(0,\beta x^4)=d((\beta,0)\cdot(0,x^4))=(0,x^4)\cdot d(\beta,0)+(\beta,0)\cdot d(0,x^4)=(0,x^4)\cdot d(\beta,0)$.
Consequently, $d(\beta,0)=d(\beta+x^4,0)+(0,x^4)d(\beta,0)$, and so $(1,x^4) d(\beta,0)=d(\beta+x^4)$, and $\pi$ is the inverse of $(1,x^4)$.
\end{rem}
The special case of $m=1$ recovers the identity $[\alpha,\beta)_{p,F}=[\alpha(\beta+x^p)\beta^{-1},\beta+x^p)_{p,F}$ appearing in \cite[Page 119, Equation (2)]{ChapmanQueguiner:2024}.

The following proposition extends \cite[Lemma 2.2]{ChapmanQueguiner:2024}
\begin{lem}\label{neat}
For any two cyclic algebras $[\omega,\beta)_{p^m,F}$ and $[\alpha,\gamma)_{p,F}$ of degree $p^m$ and $p$ respectively, there exist $\tau \in W_m F$ and $\delta \in F^\times$ for which $[\omega,\beta)_{p^m,F}=[\tau,\delta)_{p^m,F}$ and $[\alpha,\gamma)_{p,F}=[\delta,\gamma)_{p,F}$.
\end{lem}

\begin{proof}
The equation $\alpha+x^{p^m}-x=\beta+x^{p^m}$ is actually a linear equation that has a unique solution $x=\alpha-\beta$. Take this $x$, and set $\delta=\beta+x^{p^m}$ and $\tau=\omega \cdot (1+x^{p^m}\beta^{-1},0,\dots,0)\cdot \pi$, and the statement readily follows.
\end{proof}

\section{The cyclicity of tensor products}

We start with a theorem on the tensor product of a cyclic algebra of degree $p^m$ and a cyclic algebra of degree $p$.

\begin{thm}\label{main}
Every tensor product of a cyclic algebra of degree $p^m$ and a cyclic algebra of degree $p$ over a field $F$ of $\operatorname{char}(F)=p>0$ is cyclic of degree $p^{m+1}$.
\end{thm}

\begin{proof}
By Lemma \ref{neat}, the algebras can be written as $[\tau,\delta)_{p^m,F}$ and $[\delta,\gamma)_{p,F}$.
Now, $[\tau,\delta)_{p^m,F}=[\tau,\delta \gamma^{p^m})_{p^m,F}=[(0,\tau),\delta\gamma^{p^m})_{p^{m+1},F}$ and $[\delta,\gamma)_{p,F}=[(\delta,0,\dots,0),\gamma^{p^m})_{p^{m+1},F}=[(\delta,0,\dots,0),\gamma^{p^m})_{p^{m+1},F}+[(\delta,0,\dots,0),\delta)_{p^{m+1},F}=[(\delta,0,\dots,0),\delta\gamma^{p^m})_{p^{m+1},F}$.
Therefore, $[\tau,\delta)_{p^m,F} \otimes [\delta,\gamma)_{p,F}=[(0,\tau),\delta\gamma^{p^m})_{p^{m+1},F} \otimes [(\delta,0,\dots,0),\delta\gamma^{p^m})_{p^{m+1},F}=[(\delta,\tau),\delta\gamma^{p^m})_{p^{m+1},F}$.
\end{proof}

In short, we obtained the following formula
\begin{eqnarray}
[\tau,\delta)_{p^m,F} \otimes [\delta,\gamma)_{p,F}=[(\delta,\tau),\delta \gamma^{p^m})_{p^{m+1},F}.
\end{eqnarray}

\begin{cor}\label{prime}
Any tensor product of cyclic $p$-algebras of prime degree $p$ is a cyclic algebra.
\end{cor}

\begin{proof}
It follows by induction readily from Theorem \ref{main}.
\end{proof}

\begin{exmpl}
Given two cyclic $p$-algebras $[\omega,\beta)_{p,F}$ and $[\alpha,\gamma)_{p,F}$, we can write them as $[\tau,\delta)_{p,F}$ and $[\delta,\gamma)_{p,F}$ where $x=\alpha-\beta$, $\delta=\beta+x^p$ and $\tau=\omega \delta \beta^{-1}$. And then
$$[\omega,\beta)_{p,F} \otimes [\alpha,\gamma)_{p,F}=[(\delta,\tau),\delta \gamma^p)_{p^2,F}.$$
\end{exmpl}

\section{More general case}

In the more general case of tensor products of $p$-algebras of arbitrary degrees (not necessarily prime), we can provide a computational method for finding a cyclic representative of the Brauer class.

\begin{thm}[{\cite[Chapter 7, Theorem 31]{Albert:1968}}]
If $F$ is a field of $\operatorname{char}(F)=p>0$ and $A$ is a tensor product of cyclic $p$-algebras over $F$, then $A$ is Brauer equivalent to a cyclic $p$-algebra.
\end{thm}

\begin{proof}
Suppose $\exp A=p^m$. We do it by induction on $m$.
For $m=1$, $A$ is Brauer equivalent to a tensor product of cyclic algebras of degree $p$, which by Corollary \ref{prime} is cyclic.
Suppose this holds true for $m-1$.
Then consider the Brauer class of $A^{\otimes p}$.
It is of exponent $p^{m-1}$, and so it is represented by a cyclic algebra $[\omega,\beta)_{p^t,F}$ for some $t$.
Therefore, $B=A \otimes [-(\omega,0),\beta)_{p^{t+1},F}$ is of exponent dividing $p$. Hence, $B$ is Brauer equivalent to a tensor product of cyclic algebras of degree $p$. Consequently, $A$, which is Brauer equivalent to $[\omega,\beta)_{p,F} \otimes B$, is by Corollary \ref{prime} Brauer equivalent to a cyclic algebra.
\end{proof}

The proof of this theorem can be made computational, because at each step we apply Corollary \ref{prime} which is based on Theorem \ref{main} that gives an explicit formula.
\bibliographystyle{abbrv}
\bibliography{bibfile}
\end{document}